\documentclass[12pt,centertags]{amsart}
\usepackage{amsmath,amssymb,amscd,amsfonts,amsbsy,amsxtra,latexsym} 
\usepackage[mathscr]{eucal}
\usepackage{mathrsfs}
\usepackage{epsfig}
\usepackage{epstopdf}
\usepackage[utf8]{inputenc} 
\usepackage{graphicx}        
\usepackage[all]{xy}
\usepackage{tikz}
\usetikzlibrary{matrix,arrows,decorations.pathmorphing}

\usepackage{epsf}
\numberwithin{equation}{section}

\newcommand{\field}[1]{\mathbb{#1}}
\newcommand{\Z}{\field{Z}}
\newcommand{\R}{\field{R}}
\newcommand{\C}{\field{C}}
\newcommand{\N}{\field{N}}
\newcommand{\PP}{\field{P}}

\def\mM{\mathcal{M}}
\def\mO{\mathcal{O}}

\def\mT{\mathcal{T}}

\DeclareMathOperator{\Id}{Id}

\DeclareMathOperator{\tr}{Tr}

\DeclareMathOperator{\td}{Td}

\DeclareMathOperator{\ch}{ch}

\newtheorem{thm}{Theorem}[section]
\newtheorem{lemma}[thm]{Lemma}
\newtheorem{prop}[thm]{Proposition}

\theoremstyle{definition}
\newtheorem{rem}[thm]{Remark}
\theoremstyle{definition}
\newtheorem{defn}[thm]{Definition}
\newcommand{\be}{\begin{eqnarray}}
\newcommand{\ee}{\end{eqnarray}}
\newcommand{\ov}{\overline}
\newcommand{\wi}{\widetilde}

\newcommand{\comment}[1]{}

\def\q \m#1#2{{\raise1pt\hbox{$#1$}\kern-1pt\big/
               \kern-1pt\raise-1pt\hbox{$#2$}}}

\begin{document}

\title{Quillen metrics and branched coverings}
 \author{Xiaonan Ma}
\address{Universit\'e Paris Cit\'e,
CNRS, IMJ-PRG, B\^atiment Sophie Germain, UFR de Math\'ematiques, 
Case 7012, 75205 Paris Cedex 13, France}

\email{xiaonan.ma@imj-prg.fr}
\date{\today}

\begin{abstract} We give a formula to compare the Quillen 
metrics associated to a branched covering from holomorphic line 
bundles.
\end{abstract}


\maketitle
\setcounter{section}{-1}

\section{Introduction} \label{s0}

The Quillen metric is a metric on  the determinant of 
the cohomology of a holomorphic vector bundle  over a complex manifold. 
It is the product of $L^{2}$-metric and the analytic torsion which is 
the regularized determinant of the Kodaira Laplacian.
By Quillen, Bismut, Gillet  and  Soul\'e, we know that the Quillen metric has 
very nice 
proprieties.

Let $i: Y \hookrightarrow X$ be an immersion of compact complex 
manifolds. Let $\eta$ be a holomorphic Hermitian vector bundle over $Y$. 
Let $\xi$ be a holomorphic resolution of $\eta$ over $X$. 
Bismut and Lebeau \cite{BL91} have calculated the relation  of the 
Quillen metrics associated to $\eta$ and $\xi$.

Let $\pi: W \to S$ be a holomorphic  map of compact complex manifolds. 
Let $\xi$ be a holomorphic Hermitian vector bundle over $W$. 
Let $R^\bullet  \pi _*\xi$ be  the direct image of $\xi$. 
Let $\lambda(\xi)$ and $\lambda(R^\bullet  \pi _*\xi)$ be the inverses of
the determinant of  the cohomology of  $\xi$ and $R^\bullet  \pi _*\xi$. 
By \cite{KnudsenMumford},
$\lambda(\xi) \simeq \lambda(R^\bullet  \pi _*\xi)$. If $\pi$ 
is a submersion, 
Berthomieu  and Bismut \cite{BerB94} have compared the corresponding 
Quillen metrics on $\lambda(\xi)$ and $\lambda(R^\bullet  \pi _*\xi)$.

Suppose now that $W, S$ are arithmetic varieties  over Spec($\Z$). 
Let $\xi$ be an algebraic vector bundle on $W$. In  \cite{GS91}, 
Gillet and Soul\'e conjectured that 
an arithmetic Riemann-Roch formula holds. 
In  \cite{GS92}, by using the Bismut-Lebeau embedding formula
for Quillen metrics \cite{BL91}, they proved it for the first 
arithmetic Chern class. By using Bismut's  work \cite{B97}, 
the family version of \cite{BL91},  Gillet-Roessler-Soulé \cite{GRS08}
show that the arithmetic Riemann-Roch formula in higher degrees holds.

In \cite{B94}, Bismut has  conjectured   an equivariant arithmetic 
 Riemann-Roch  formula. In \cite{B95}, he was able to show
 the compatibility of his conjecture with immersions. 
In \cite{KRo01}, K\"ohler and Roessler have obtained a version of 
Bismut's  conjecture by using  \cite{B95}.
For more recent works in this direction, cf.
\cite{KingsRo17, MRo09,Tang21}.

In this paper, we will compare the Quillen 
metrics on $\lambda(\xi)$ and $\lambda(R^\bullet  \pi _*\xi)$
in the case that $\pi$  is a branched covering from a holomorphic 
line bundle.
For any holomorphic line bundle over a compact K\"ahler manifold $S$, 
we give a general construction of a smooth submanifold 
$W\subset L$ (cf. (\ref{eq:1.1})) from holomorphic sections of the 
powers of $L$ on $S$ such that $\pi_{W}: W\to S$ the projection from 
$W$ on $S$, is a branched covering. We obtain the analogue of
the result of Berthomieu-Bismut \cite[Theorem 0.1]{BerB94}
and its equivariant version  \cite[Theorem 3.1]{Ma002}
in this situation.
In fact, our first  result, Theorem \ref{t3.2}  is compatible with 
the arithmetic Riemann-Roch formula. 
Our  second result, Theorem  \ref{t4.1} fits
perfectly well with Bismut's  conjecture.

This paper is organized  as follows. 
In Section \ref{s1}, we construct a 
branched covering  from a holomorphic line bundle. In Section \ref{s2}, 
we describe  the canonical sections of determinant lines. 
In Section \ref{s3}, applying the Bismut-Lebeau embedding formula
\cite[Theorem 0.1]{BL91},  we calculate 
the Quillen norm of canonical section.
In Section \ref{s4}, using  the Bismut equivariant embedding formula 
\cite[Theorem 0.1]{B95}, we calculate 
the equivariant Quillen norm of the canonical section. 

\noindent{\bf Acknowledgments}.  I'm very much indebted to Professor 
Jean-Michel Bismut for very helpful discussions and suggestions.
This paper was written some time ago, it has remained unpublished.
The author is indebted to Esteban Gomezllata Marmolejo
who informs us that our result is useful for his thesis \cite{Gom22}.
 This motives us to deliberate it.

\section{Branched coverings}\label{s1}

{$\quad $} Let $S$ be a compact complex  
manifold. 
Let $L$  be a holomorphic line bundle on $S$. 

Let $\alpha_i \in H^0(S,L^i)$ ($1\leq i\leq d$, $d\geq 2, d\in   \N^{*}$).
For $(x,t) \in L$, $x \in S$, set
\begin{align}\label{eq:1.1}\begin{split}
F(\alpha)(x,t)&= t^d + \sum_{i=1}^d \alpha_i (x)  t^{d-i},\\
W&= \Big \{ (x,t)\in L:  F(\alpha)(x,t) = 0 \Big \}.
\end{split}\end{align}
We suppose  that $W$ is smooth.

Let $V= \PP(L\oplus 1)$ the projectivisation of the vector bundle 
$L\oplus \C$, here $\C$ is the trivial line bundle on $S$.
We identify $S$ with $\{ (x, (0,1))\in V: x\in S\}\subset V$.
Let  $\pi: V\to S$ be the natural projection with fibre $Y$. 
The complement of $\PP(L)\simeq S$ in $V$ is canonical isomorphic   to $L$, 
so we can identify $W$ to  a sub-manifold of $V= \PP(L\oplus 1)$. 
Let  $\pi_W : W\to S$ be 
the projection induced by $\pi$. Then $W$ is a branched  covering of $S$ 
of degree $d$. 

Let $\xi$ be a holomorphic vector  bundle on $S$. Let
\begin{align}\label{eq:1.1a}
	\xi'= \pi^*_W \xi
\end{align}	
be the pull-back of the bundle $\xi$ on $W$.
Let $R^\bullet  \pi _{W*}\xi'$,
 $R^\bullet  \pi _{W*} \mO _W$ be the direct images of $\mO _W(\xi')$, 
$\mO _W$, the sheaves of holomorphic sections of $\xi'$,
and of holomorphic functions on $W$, respectively.
By  \cite[Theorem 2.4.2]{GrauertRemmert84},
 $R^\bullet  \pi _{W*} \mO _W = R^0  \pi _{W*} \mO _W$
is locally free of rank $d$ on $S$. By \cite[Exercise 3.8.3]{Hartshorne77}, 
we have 
\begin{align}\label{eq:1.2}
R^\bullet  \pi _{W*}\xi'= R^0  \pi _{W*} \mO _W \otimes \xi.
\end{align}
Let $H^\bullet(W,\xi')=\bigoplus_{j=1}^{\dim W} H^{j}(W,\xi')$, 
$H^\bullet(S, R^0  \pi _{W*}\xi')$
be the cohomology groups  of $\mO _W(\xi')$ on $W$, 
$\mO_{S}(R^0  \pi _{W*}\xi')$ on $S$, respectively.

For a complex vector space  $E$, the determinant line of $E$ is 
the complex line
\begin{align}\label{eq:1.2a}
	\det E= \Lambda^{\text{max}}E.
\end{align}	

\begin{defn}\label{d1.1} Set
\begin{align}\label{eq:1.3}\begin{split}
\lambda(\xi')&=\bigotimes_i \Big(\det H^i(W,\xi')\Big)^{(-1)^{i+1}},\\
\lambda(R^\bullet  \pi _{W*}\xi')&= \bigotimes_i 
\Big(\det H^i(S, R^0  \pi _{W*}\xi')\Big)^{(-1)^{i+1}}.
\end{split}\end{align}
\end{defn}
By \cite{KnudsenMumford}, we have the canonical isomorphism 
$\lambda(\xi')\simeq \lambda(R^\bullet  \pi _{W*}\xi')$.
Let $\sigma$ be the canonical section of 
$\lambda(\xi') \otimes \lambda^{-1} (R^\bullet  \pi _{W*}\xi')$.

{\bf Example}: Let $\C\PP ^n$ be the complex projective space 
of dimension $n$. 
Let  $(z_0, \cdots, z_n) = (z_0, z)$ be the  homogeneous  coordinate.
 Let $S= \{z_0=0\}$ $\hookrightarrow    \C \PP ^n$.
 Let $W$ be a hypersurface of degree $d$ which doesn't  contain 
the point $(1,0)$. Let $\pi : W\to S$ be the projection from $(1,0)$.
Let $L= \mO _S (1)$ be the hyperplane line bundle on $S$.
By \cite[p167]{GriffithsHarris}, we can reduce this 
to the situation (\ref{eq:1.1}).

\begin{rem}\label{r1.2} 
	Let $\pi: S_1\to S_2$ be a finite mapping of Riemann 
surfaces of degree $n$. Let $\mM (S_1), \mM (S_2)$ be the meromorphic
 function fields on $S_1, S_2$. Then $\pi$ caracterize by the finite field
 extension $\mM (S_2) \hookrightarrow \mM (S_1)$ 
 \cite[\S 2.11]{Shokurov94}. 
So $S_1, \pi$ is constructed   by an irreducible polynomial 
\begin{align}\label{eq:1.4}
P(T)= T^n + c_1 T^{n-1} + \cdots + c_n \in \mM (S_2) (T).
\end{align}
So our construction contains a large part of general map of 
Riemann surfaces.
\end{rem}

Let $\imath: S\to V$, $\jmath: W\to V$ be the natural immersions.
  Let $\mO _V(-1)$ be the universal 
line bundle over $V$. Let $\mO _V(k)= \mO _V(-1)^{\otimes -k}$.
On $V$, we have the exact sequence of holomorphic vector bundles
 \cite[(1.21)]{B97b},
\begin{align}\label{eq:1.5}
0 \to \mO _V(-1)  \stackrel{a}{\to}  \pi^* L \oplus \C  \stackrel{a}{\to} 
\frac{\pi^* L \oplus \C}{\mO _V(-1) } \to 0.
\end{align}
Let $\tau_{[S]}(y)\in \Big(\frac{\pi^* L \oplus \C}{ \mO _V(-1) }\Big)_{y}$ 
be given by
\begin{align}\label{eq:1.6}
\tau_{[S]}(y) = a_{y}(0,-1).
\end{align}
Then $\tau_{[S]}$ is a holomorphic section of 
$\frac{\pi^* L \oplus \C}{\mO _V(-1) }$ which vanishes exactly on $S$.
The map $\theta: \pi^* L \to \frac{\pi^* L \oplus \C}{\mO _V(-1) }$
 induced by the projection from $\pi^* L \oplus \C$ is an isomorphism on 
$L\subset \PP(L\oplus 1)$.
 Under this identification, $\tau_{[S]}$ is the tautological
 section of $\pi^* L $ on $L$.
We have 
\begin{align}\label{eq:1.6a}
\mbox{div} (\tau_{[S]}) = {S}. 
\end{align}
Let $\sigma_{[S]}$ be the canonical section of $[S]$
on  $\PP(L\oplus 1)$. 
Then $\sigma_{[S]}^{-1} \otimes \tau_{[S]}$ 
is  a nonzero section of 
$ [S]^{-1} \otimes \frac{\pi^* L \oplus \C}{\mO _V(-1) }$.
We identify the line bundle $[S]$ to 
$\frac{\pi^* L \oplus \C}{\mO _V(-1) }$ via this section.
In particular, we get
\begin{align}\label{eq:1.7a}
	[S] |_{S}= L.
\end{align}
The exact sequence \eqref{eq:1.5}  induces also an isomorphism 
\begin{align}\label{eq:1.7}
[S] \simeq \pi^* L \otimes \mO _V (1).
\end{align}

\begin{rem}\label{r1.3}  If the linear system $|L^d|$ hasn't any base points,
 then for the generic elements
$\alpha_i \in H^0(S,L^i)$ ($1\leq i\leq d$), $W$ is smooth.

In fact, let $\nu$ be the holomorphic section of $\mO _V(1)$ defined 
by $(0,1) \in (\pi^* L\oplus \C)^*$, then 
\begin{align}\label{eq:1.8}
\mbox{div} [\nu] = \PP(L). 
\end{align}
By \eqref{eq:1.6}, for $c\in   \C $, 
$\alpha_i \in H^0(S,L^i)$ ($1\leq i\leq d$), put
\begin{align}\label{eq:1.9}
G(\alpha, c) = c \tau_{[S]}^d + 
\sum_{i=1}^d \alpha_i(x) \nu^i \tau_{[S]}^{d-i},
\end{align}
then $\{ G(\alpha, c): \alpha_i \in H^0(S,L^i),  
1\leq i\leq d, c\in   \C  \}$ is a linear system of $[dS]$ on $V$,
 and the base locus of this system is empty.
By Bertini's Theorem  \cite[p137]{GriffithsHarris}, 
$\{G(\alpha,1)=0\}\subset \PP(L\oplus 1)$ is smooth for generic 
elements $\alpha_i \in H^0(S,L^i)$.
If we identify $\pi^* L$ to $[S]$ on $L$ 
as above, then $G(\alpha,1)= F(\alpha)$, so we obtain our Remark.
\end{rem}

 \section{Canonical isomorphisms of determinant lines}\label{s2}

{$\quad$} By \eqref{eq:1.1}, we can extend $F(\alpha)$ to  
a meromorphic section of $\pi^* L^d$ on $V$.
Let $t: L \to \pi^* L$ be the tautological section of $\pi^* L$ 
on $L\subset \PP(L\oplus 1)=V$. Then $t$ extends naturally 
to a meromorphic section of $\pi^* L$ on $V$. Set 
\begin{align}\label{eq:2.1}\begin{split}
f(\alpha) = F(\alpha) /t^d  .
\end{split}\end{align}
Then $f(\alpha)$ is  a meromorphic function on $V$, and 
\begin{align}\label{eq:2.2}
\mbox{div} (f(\alpha)) =  W- d \cdot S.
\end{align}
Let $\delta_{\{W\}}, \delta_{\{S\}}$ be the currents on $V$ defined by 
the integration on $W,S$.
By \eqref{eq:2.2}, we have 
\begin{align}\label{eq:2.3}
\frac{ \overline{\partial} \partial}{2i\pi} \log |f(\alpha)|^2 
= \delta_{\{W\}} - d\, \delta_{\{S\}}.
\end{align}
We will identify the line bundle $[W]$ to $[dS]$ via $f(\alpha)$. 
Let $\tau_{[W]}$ be the canonical section of $[W]$ on $V$, then 
\begin{align}\label{eq:2.4}
\tau_{[W]}= f(\alpha)\tau_{[S]}^d.
\end{align}

Let $TY= TV/S$ be the holomorphic tangent bundle to the fibre $Y$. 
By \eqref{eq:1.4}, as in \cite[p409]{GriffithsHarris}, 
we have  an exact sequence of holomorphic vector bundles on $V$, 
\begin{align}\label{eq:2.5}
0 \to   \C  \to (\pi^* L \oplus \C) \otimes \mO _V (1) \to TY \to 0.
\end{align}
Let $K_Y= T^*Y $ be the relative canonical bundle on $V$. 
By \eqref{eq:2.5}, 
\begin{align}\label{eq:2.6}
K_Y \simeq \pi^* L^{-1} \otimes \mO _V (-2).
\end{align}

\begin{prop} \label{p2.1} For $k>0$, we have canonical identifications
 \begin{align} \label{eq:2.7}\begin{split}
&R^0 \pi_* \mO _V= \C,\qquad \qquad 
R^0 \pi_* \mO _V (-k)=0,\\
&R^1 \pi_* \mO _V (-1)= R^1 \pi_* \mO _V =0,\\
&R^1 \pi_* \mO _V (-(k+1))= \bigoplus_{i=1}^{k} L^i.
\end{split}\end{align}
\end{prop}
\begin{proof}  The first two equations are  trivial.

Using the Serre duality  \cite[p240]{Hartshorne77} and \eqref{eq:2.6}, 
for $m\in \Z$, we have 
\begin{align} \label{eq:2.8}\begin{split}
R^1 \pi_* \mO _V (-m)
\simeq (H^0(Y,  \mO _V (m) \otimes K_Y))^*\\
\hspace*{15mm} = L\otimes  (H^0(Y,  \mO _V (m-2)))^*.
\end{split}\end{align}
The second equation of (\ref{eq:2.7}) and (\ref{eq:2.8})
imply the third equation of (\ref{eq:2.7}).

For $k>0$, by \cite[p165]{GriffithsHarris}, we have 
\begin{align}\label{eq:2.9}
H^0(Y,  \mO _V (k-1)) = \mbox{Sym}^{k-1} ((L\oplus \C)^*) 
= \bigoplus_{i=0}^{k-1} L^{-i}.
\end{align}
By \eqref{eq:2.8} and \eqref{eq:2.9}, we get the last equation of 
\eqref{eq:2.7}.
\end{proof}

\begin{prop} \label{p2.2} We have a canonical  isomorphism,
\begin{align}\label{eq:2.10}
R^0\pi_{W*} \mO _W \simeq \bigoplus _{j=0}^{d-1} L^{-j}.
\end{align}
\end{prop}
\begin{proof} 
By \cite[\S 2.4]{GrauertRemmert84}, we can identify 
$R^0\pi_{W*} \mO _W$  as the sheaf of polynomial functions along the 
fiber $L$ with degree $\leq d-1$, thus we get 
(\ref{eq:2.10}).
\end{proof}

Using \cite[Exercise 3.8.3]{Hartshorne77}, \eqref{eq:1.7}, 
\eqref{eq:2.7} and \eqref{eq:2.10}, for $k \geq 2$, we get
\begin{align}\label{eq:2.12} \begin{split}
R^{\bullet} \pi_* \mO _V ([-S])& =0,\\
R^0 \pi_* \mO _V ([-kS])& =0, \quad 
R^1 \pi_* \mO _V ([-kS]) = \bigoplus_{j=1}^{k-1} L^{-j},\\
R^{0}\pi_{W*} \xi' & \simeq  \bigoplus_{j=0}^{d-1} L^{-j}\otimes \xi.
\end{split}\end{align}
Note that we identify $[W]$ with $[dS]$ via (\ref{eq:2.4}), by
(\ref{eq:2.12}), we have 
\begin{align}\label{eq:2.14a}
R^0\pi_* \mO _V ([-W])= 0, \quad 
R^1\pi_* \mO _V ([-W])= \bigoplus_{j=1}^{d-1} L^{-j}.
\end{align}

We have the following exact sequence of sheaves over $V$
\begin{align} \label{eq:2.13}
0 \to \mO _V([-W]) \stackrel{\tau_{[W]}}{\to} \mO _V 
\to \jmath_* \mO _W \to 0.
\end{align}
By \eqref{eq:2.7}, \eqref{eq:2.13}, we get the following exact 
sequence of sheaves on $S$
\begin{align}\label{eq:2.14}
 0\to  R^0 \pi_* \mO _V \stackrel{\jmath}{\to}  R^0 \pi_{W*} \mO _W 
 \stackrel{\delta_1}{\to}
 R^1\pi_* \mO _V ([-W]) \to  0.
\end{align}

\begin{prop} \label{p2.3} Under the canonical identification
	\eqref{eq:2.10}, 	\eqref{eq:2.14a}, 
	the exact sequence \eqref{eq:2.14} is canonically  split.
 Let $\delta: \bigoplus_{i=1}^{d-1} L^{-i} \to  R^1\pi_* \mO _V ([-W])$ 
be the map  induced by $\delta_1$ and \eqref{eq:2.10}, then   
under the decomposition $L^{-1}\oplus \cdots \oplus L^{-d+1}$,
we have 
\begin{align} \label{eq:2.15}
\delta^{-1} = (a_{ij})=\left (  \begin{array}{lll}
1& &{*}\\
&\ddots&\\
  0&&1 \end{array}\right ).
\end{align}
Moreover  
\begin{align} \label{eq:2.16}\begin{split}
a_{ij}= \alpha_{j-i} \quad \mbox{if} \quad j>i,\\
\hspace*{11mm} 1 \quad \mbox{if}\quad  i=j,\\
\hspace*{11mm} 0 \quad \mbox{if}\quad  j<i.
\end{split}\end{align}
\end{prop}

\begin{proof} Clearly, under the identification \eqref{eq:2.7}, 
$\jmath$ is the canonical embedding of   $\C$  into the factor    $\C$  
in $R^0 \pi_{W*} \mO _W $,
so the exact sequence \eqref{eq:2.14} is canonical split.

To prove \eqref{eq:2.15}, we use  \v{C}ech cohomology.
Before prove \eqref{eq:2.15}, we explain the compatibility 
of \eqref{eq:2.8}, \eqref{eq:2.9} and \v{C}ech 
 cohomology on $\C\PP ^n$. 

Let $(X_0,\cdots, X_n)$ be linear coordinates on $  \C ^{n+1}$, 
and let $\{x_i= X_i/X_0:  i=1, \cdots, n\}$ be the 
corresponding affine coordinates. 
Let $U_i= (X_i\neq 0) \subset     \C\PP ^n$.
Let $K= \Lambda ^n (T^*    \C\PP ^n)$ be the canonical line  bundle on 
$   \C \PP ^n$. By \cite[p409]{GriffithsHarris}, we have an exact sequence of
 holomorphic vector bundles on $    \C\PP ^n$
\begin{align}\label{eq:2.17}
0 \to   \C  \to \mO _{\C\PP^n} (1)^{n+1}
\to T    \C\PP ^n\to 0.
\end{align}
By \eqref{eq:2.17}, we have 
\begin{align}\label{eq:2.18}
K \stackrel{v}{\simeq} \mO _{\C\PP^n} (-(n+1)).
\end{align}
We trivialize $\mO _{\C\PP^n}(1)$ by 
$(1, 0,\cdots, 0)\in   \C ^{n+1, *}$
 on $U_0$. By \cite[p409]{GriffithsHarris}, on $U_0$, we have 
\begin{align}\label{eq:2.19}
v(dx_1 \wedge \cdots \wedge dx_n) = 1 
\in \mO _{\C\PP^n}(-(n+1)).
\end{align}
By \cite[Remark 3.7.1.1]{Hartshorne77}, there exists a  canonical element 
$a \in H^n(    \C\PP ^n, K)$ which defines the Serre duality $\mu$.
 On $\cap_{i=0}^n U_i$,
consider the  cocycle, we have  
\begin{align}\label{eq:2.20}
a = \frac{1}{x_1\cdots x_n} dx_1 \wedge \cdots \wedge dx_n.
\end{align}
By \cite[Theorem 3.5.1]{Hartshorne77}, using \v{C}ech  cohomology, 
on $\cap_{i=0}^n U_i$,
$H^n(    \C\PP ^n, \mO _{  \C\PP ^n}(-n-k-1))$ 
$(k\in   \N )$ is generated by  the \v{C}ech  cocycle 
$$ \Big\{ \alpha_{l_1\cdots l_n}= x_1^{-(l_1+1)} \cdots x_n^{-(l_n+1)}:
\sum_{i=1}^n l_i \leq k, l_i\in   \N \Big \}.$$
Also  $H^0(   \C \PP ^n, \mO _{  \C\PP ^n}(k))$
$(k\in   \N )$ is generated by 
$$\Big\{ \beta_{l_1\cdots l_n}= x_1^{l_1} \cdots  x_n^{l_n}: 
 \sum_{i=1}^n l_i \leq k, l_i\in   \N \Big\}$$ 
on $U_0$. By \eqref{eq:2.19}, \eqref{eq:2.20},
 we have the following commutative diagram 
\begin{align} \label{eq:2.21}	
	\begin{split}
	\xymatrix{
H^n(\C\PP ^n, \mO _{\C\PP ^n}(-n-k-1))\ar[d]^{\mu} 
\ar[r]^{\text{\v{C}ech}}&
\text{span} \Big\{\alpha_{l_1\cdots l_n}:
\sum_{i=1}^n l_i \leq k, l_i\in   \N  \Big\}
\ar[d]^{\mu_{1}} \\
 H^0(\C\PP ^n, \mO _{ \C \PP ^n}(n+k+1)\otimes K)^* 
\ar[d]^{\wr \wr v} & \\
 {H^0(\C\PP ^n,} \mO _{\C  \PP ^n}(k))^* 
 \ar[r]^{\text{\v{C}ech}}&\quad
 { \Big(\text{span} \Big\{\beta_{l_1\cdots l_n}:
 \sum_{i=1}^n l_i \leq k, l_i\in   \N  \Big\} \Big)^*}.
}    \end{split}  
\end{align}
Thus the map $\mu_1$ is  such that  
\begin{align}
\mu_1(\alpha_{l_1\cdots l_n}) (\beta_{l'_1\cdots l'_n}) 
= \delta_{(l_1\cdots l_n), (l'_1\cdots l'_n)}. \nonumber
\end{align}

Now we are ready to establish \eqref{eq:2.15}.
Let $(v,u)$ be the local homogeneous coordinates of $\PP(L\oplus 1)$. 
Let $U_1 = \{ (v,u)\in \PP(L\oplus 1): v \neq 0 \}$, 
$U_2 = \{ (v,u)\in \PP(L\oplus 1): u \neq 0 \}$ with affine 
coordinate $t$ as a function on $U_{2}$ with values in $\pi^{*}L$. 
We will identify $U_2$ with $L$, then 
$U_1 \cap U_2= L\setminus S$, $t^{-1}$
is a section of $\pi^* L^{-1}$ on $U_1 \cap U_2$. 
As explained in the proof of Proposition \ref{p2.2}, 
on $U_2$, for $x\in S$, we have
\begin{align}\label{eq:2.22}
(R^0\pi_{W*} \mO _W)_x = \Big\{\sum_{i= 0}^{d-1} \gamma_i t^i :  \quad 
\gamma_i \in \mO _{S,x} (L^{-i})  \Big\}.
\end{align}

We recall that  from  \eqref{eq:1.7}, \eqref{eq:2.2}, on $V$, 
\begin{align}
[-W] = \pi^* L^{-d} \otimes \mO _V(-d). \nonumber
\end{align}
By  \eqref{eq:2.7}, \eqref{eq:2.21}, on $U_1 \cap U_2$,  for $x\in S$,  
we have 
\begin{align}\label{eq:2.23}
 (R^1\pi_* \mO _V ([-W]))_x 
 = \Big\{\sum_{j=1}^{d-1} \gamma_{d-j} t^{-j} : 
\quad \gamma_{d-j} \in \mO _{S,x} (L^{-d+j}) \Big\}.
\end{align}

On $U_1 \cap U_2$, we have 
\begin{align} 
\tau_{[W]} (\gamma_{d-j} t^{-j})= \gamma_{d-j} \Big(t ^{d-j}
+ \sum_{i=1}^{d} \alpha_i(x) t^{d-i-j} \Big). \nonumber
\end{align}
The function $\gamma_{d-j} (t^{d-j}
+ \sum_{i=1}^{d-j-1} \alpha_i(x) t^{d-i-j})$ is holomorphic 
on $U_2$, $\gamma_{d-j} (\sum_{i=d-j}^{d} \alpha_i(x) t^{d-i-j})$
is holomorphic on $U_1$. By the definition of $\delta_1$, we have 
\begin{align}\label{eq:2.24}
\delta_1 \Big (\gamma_{d-j} (t^{d-j}
+\sum_{i=1}^{d-j-1} \alpha_i(x) t^{d-i-j} )\Big )
= \gamma_{d-j} t^{-j}.
\end{align}

By  \eqref{eq:2.22},  \eqref{eq:2.23} and \eqref{eq:2.24},  
we have \eqref{eq:2.15} and \eqref{eq:2.16}.    \end{proof}

We  also have   an exact sequence of sheaves  over $V$
\begin{align} \label{eq:2.25}\begin{split}
0 \to \mO _V([-dS]) \stackrel{\tau_{[S]}^d}{\to}
 \mO _V \to \imath_* \mO _S\big(\bigoplus_{i=0}^{d-1}L^{-i}\big) \to 0.
\end{split}\end{align}
By \eqref{eq:2.7} and \eqref{eq:2.25}, we have the following  
 exact sequence of sheaves over $S$
\begin{align}\label{eq:2.26} \begin{split}
0\to R^0 \pi_* \mO _V \to 
\mO _S\big(\bigoplus_{i=0}^{d-1} L^{-i}\big) \stackrel{\delta'}{\to}
R^1\pi_* \mO _V ([-dS]) =\mO _S\big(\bigoplus_{i=1}^{d-1}L^{-i}\big)\to 0.
\end{split}\end{align}

\begin{prop} \label{p2.4} Under the identifications \eqref{eq:2.12},  
the exact sequence \eqref{eq:2.26} is naturally split, and  
${\delta'}|_{\bigoplus_{i=1}^{d-1} L^{-i}}= \Id$.
\end{prop}

 \begin{proof} We use the notation in the proof of Proposition \ref{p2.3}.
On $U_1 \cap U_2$, we have 
\begin{align}\label{eq:2.27}
\tau_{[S]}^{d} (\gamma_{d-j} t^{-j})= \gamma_{d-j} t^{d-j}.
\end{align}
In \eqref{eq:2.25} as in (\ref{eq:2.23}), we have identified 
$\imath_* \mO _S(L^{-i})$ to 
$\{\gamma_i t^i: \gamma_i\in \mO _S(L^{-i})\}$ on $U_2$.

By the definition of $\delta'$ and (\ref{eq:2.27}), we have 
  Proposition \ref{p2.4}.\end{proof}

As in Definition \ref{d1.1}, we define the  complex lines
\begin{align} \label{eq:2.28}\begin{split}
\lambda'_d (\pi^* \xi)&= \lambda(\pi^* \xi) \otimes 
\lambda ^{-1}([-dS]\otimes \pi^* \xi),\\
\lambda _W(\xi)&= \lambda (R^0 \pi_* \mO _V\otimes \xi)
\otimes \lambda (R^1\pi_*([-W]) \otimes \xi).
\end{split}\end{align}
By \cite{KnudsenMumford}, \eqref{eq:2.13}, \eqref{eq:2.25},  
we have the canonical isomorphisms:
\begin{align}\label{eq:2.29} \begin{split}
\lambda(\xi') \simeq \lambda'_d  (\pi^* \xi),\quad 
 \lambda(\bigoplus_{i=0}^{d-1} L^{-i} \otimes \xi)
 \simeq \lambda'_d (\pi^* \xi).
\end{split}\end{align}
Let $\tau_d, \sigma_1$ be the canonical sections of 
 $\lambda^{-1}(\bigoplus_{i=0}^{d-1} L^{-i} \otimes \xi) \otimes 
\lambda '_d (\pi^* \xi)$ via (\ref{eq:2.25}),
$\lambda^{-1}(\xi')\otimes \lambda'_d  (\pi^* \xi)$
 via (\ref{eq:2.13}). Recall that $\sigma$ is the canonical section of 
 $\lambda(\xi')\otimes \lambda^{-1}  (R^{\bullet}\pi_{W*} \xi')$.

\begin{prop} \label{p2.5}  Under the identifications \eqref{eq:2.12}, 
	we have 
\begin{align} \label{eq:2.30}
\sigma = \sigma_1^{-1} \otimes  \tau_d.
\end{align}
\end{prop}
 \begin{proof} Let $\nu_{3}$ be the canonical section of 
	 $$\lambda (R^{1}\pi_{*} [-W] \otimes \xi)
\otimes 	 \lambda (R^{1}\pi_{*} [-dS] \otimes \xi)^{-1}$$
induced by $\delta$ in Proposition \ref{p2.3}. 
Let	 $p_{r}: 
	 \bigoplus_{i=r}^{d-1} L^{-i} \to\bigoplus_{i=r+1}^{d-1} L^{-i}$
	 be the canonical projection. 
	 Let $\delta_{r}: \bigoplus_{i=r}^{d-1} L^{-i}
	 \to  \bigoplus_{i=r}^{d-1} L^{-i}$ be the map defined by
	 the matrix $(a_{ij})$ as in (\ref{eq:2.16}), then we have
\begin{align} \label{eq:2.30a}	
	\begin{split}
	\xymatrix{
0\ar[r]  &   L^{-r}\ar[r]  \ar[d]^{\Id}&
\bigoplus_{j=r}^{d-1} L^{-j} \ar[r]^{p_{r}} \ar[d]^{\delta_{r}} &
\bigoplus_{j=r+1}^{d-1} L^{-j}\ar[r]\ar[d]^{\delta_{r+1}}&  0\\
0\ar[r] &    L^{-r}\ar[r] &
\bigoplus_{j=r}^{d-1} L^{-j}\ar[r]^{p_{r}}  &
\bigoplus_{j=r+1}^{d-1} L^{-j}\ar[r]&  0.
}    \end{split}  
\end{align}
By considering the long exact sequence from 	 (\ref{eq:2.30a}), 
\begin{align} \label{eq:2.31a}	\begin{split}
	\xymatrix{
0\ar[r]  &  H^{0}(S,  L^{-r}\otimes \xi)\ar[r]  \ar[d]^{\Id}&
H^{0}(S,\bigoplus_{j=r}^{d-1}  L^{-j}\otimes \xi)
\ar[r]^{p_{r}} \ar[d]^{\delta_{r}} &
H^{0}(S,\bigoplus_{j=r+1}^{d-1}  L^{-j}\otimes \xi)
\ar[r]\ar[d]^{\delta_{r+1}}&  \cdots\\
0\ar[r] &   H^{0}(S,  L^{-r}\otimes \xi)\ar[r] &
H^{0}(S,\bigoplus_{j=r}^{d-1}  L^{-j}\otimes \xi)\ar[r]^{p_{r}}  &
H^{0}(S,\bigoplus_{j=r+1}^{d-1}  L^{-j}\otimes \xi)\ar[r]&  \cdots,
}    \end{split}  
\end{align}
	as $\delta_{d-1}: L^{-d+1}\to L^{-d+1}$ is the identity map, by 
	recurrence, we know the canonical section of
	$$ \lambda (\bigoplus_{j=r}^{d-1}  L^{-j}\otimes \xi)
	\otimes \lambda^{-1} (\bigoplus_{j=r}^{d-1}  L^{-j}\otimes \xi)$$
	 induced by $\delta_{r}$ is $1$  for all $r\geq 1$.
We conclude in particular that
\begin{align}\label{eq:2.32a}
\nu_{3}=1.
\end{align}

As in (\ref{eq:1.3}), we define the complex line $\lambda(\xi)$ for 
$\xi$ on $S$. By Proposition \ref{p2.3}, (\ref{eq:2.14}) and 
(\ref{eq:2.27}), we have the following commutative diagram
\begin{align} \label{eq:2.33a}	\begin{split}
	\xymatrix{
0\ar[r]  & R^{0}\pi_{*}\mO_{V}\ar[r]  &
 R^{0}\pi_{*}\jmath_{*}\mO_{W}
\ar[r]^{\delta_{1}}  &
R^{1}\pi_{*}\mO_{V}([-W])
\ar[r]&  0\\
0\ar[r] &    R^{0}\pi_{*}\mO_{V}\ar[r]\ar[u]^{\Id} &
\mO_{S}(\bigoplus_{j=0}^{d-1}  L^{-j})\ar[r]\ar[u]^{\Id}  &
R^{1}\pi_{*}\mO_{V}([-dS])\ar[r] \ar[u]^{\delta}&  0,
}    \end{split}  
\end{align}

Let $\nu_{1}, \nu_{2}$ be the canonical sections of 
$$\lambda^{-1} (R^0 \pi_{W*} \xi') \otimes \lambda _W(\xi), \quad
\lambda^{-1} (\bigoplus_{j=0}^{d-1}  L^{-j}\otimes \xi)
\otimes \lambda(\xi)\otimes \lambda( R^1\pi_* \mO _V ([-dS])\otimes \xi)
$$
induced by  \eqref{eq:2.14} and \eqref{eq:2.26}. 
By \eqref{eq:2.15}, \eqref{eq:2.32a} and  \eqref{eq:2.33a},  we have 
\begin{align} \label{eq:2.32}
\nu_{1} = \nu_{2}\otimes \nu_{3}= \nu_{2} .
\end{align}

In our situation, the Leray spectral sequences \cite[\S 3.7]{Grot57}
associated to $\pi: V\to S$ and the considering vector bundles $\eta$
($\eta = [-dS]\otimes \pi^* \xi$, etc),  are degenerate, 
as $R^{0}\pi_{*} \eta=0$ or $R^{1}\pi_{*} \eta=0$, so 
\begin{align}\label{eq:2.33}
H^k(V, \eta) \simeq  \bigoplus_{i+j=k} H^i (S, R^j \pi_* \eta).
\end{align}
Then by \eqref{eq:2.13}, \eqref{eq:2.14}, \eqref{eq:2.26}
and \eqref{eq:2.33}, we have the following commutative diagram
of long exact sequences
\begin{align} \label{eq:2.34a}	\begin{split}
	\xymatrix{
0\ar[r]  &  H^{0}(V,  \pi^{*} \xi)\ar[r]  \ar[d]&
H^{0}(W,\pi_{W}^{*}\xi)
\ar[r] \ar[d] &
H^{1}(V, [-W]\otimes \pi^{*} \xi)
\ar[r]\ar[d]&  \cdots\\	
	0\ar[r]  &  H^{0}(S,   \xi)\ar[r]  &
H^{0}(S,R^{0}\pi_{W*}\mO_{W}\otimes \xi)
\ar[r]^{\delta_{1}} &
H^{0}(S, R^{1}\pi_{*}[-W]\otimes \xi)
\ar[r]&  \cdots\\
	0\ar[r]  &  H^{0}(S,   \xi)\ar[r] \ar[u]^{\Id} &
H^{0}(S,\bigoplus_{j=0}^{d-1}  L^{-j}\otimes \xi)
\ar[r]^{\delta'} \ar[u] &
H^{0}(S, R^{1}\pi_{*}[-dS]\otimes \xi)
\ar[r]\ar[u]^{\delta}&  \cdots\\
0\ar[r]  &  H^{0}(V,  \pi^{*} \xi)\ar[r]  \ar[u]&
H^{0}(S,\bigoplus_{j=0}^{d-1}  L^{-j}\otimes \xi)
\ar[r] \ar[u]^{\Id} &
H^{1}(V,[-dS]\otimes \pi^{*} \xi)
\ar[r]\ar[u]&  \cdots .
}    \end{split}  
\end{align}

Let $\tau$ be  the  canonical section of 
${\lambda'_d} ^{-1}(\pi^*\xi)\otimes \lambda _W(\xi)$ 
induced by \eqref{eq:2.33}. 
The $\sigma$ (resp. $\tau$) is obtained from the second vertical map 
(resp. the rest part of the vertical maps) of the 
first two lines  of \eqref{eq:2.34a}. 
The $\sigma_{1}$ (resp. $\nu_{1}$) is obtained from the first (resp. 
second) line of \eqref{eq:2.34a}, 
and $\nu_{2}$, $\tau_{d}$ is obtained from the third,
fourth line of \eqref{eq:2.34a}.  Finally $\tau$ is also obtained 
from the first and third  vertical maps of 
the last two lines of \eqref{eq:2.34a}. 

By \cite[(1.3)]{BGS88a}, 
\cite[Proposition 1]{KnudsenMumford}, 
\eqref{eq:2.13}, \eqref{eq:2.25}  and 
\eqref{eq:2.34a},  we have 
\begin{align}\label{eq:2.34}
\begin{split}
&\sigma\otimes  \tau = \sigma_1^{-1} \otimes  \nu_{1},\\
& \tau = \nu_{2} \otimes \tau_d ^{-1}.
\end{split}\end{align}

By \eqref{eq:2.32} and  \eqref{eq:2.34}, we have \eqref{eq:2.30}.
\end{proof}

For $0\leq i \leq d-1$,  we have an  exact sequence of sheaves  over $V$
\begin{align}\label{eq:2.35} \begin{split}
0 \to \mO _V([-(i+1)S]) \stackrel{\tau_{[S]}}{\to}
 \mO _V([-iS]) \to \imath_* \mO _S(L^{-i}) \to 0.
\end{split}\end{align}
By \eqref{eq:2.35}, we have the   exact sequence of sheaves  over $V$
\begin{align} \label{eq:2.36}
 \begin{split}
0 \to \bigoplus_{i=0}^{d-1} [-(i+1)S] \otimes
\pi^* \xi \stackrel{\tau_{[S]}}{\to}
 \bigoplus_{i=0}^{d-1} [-iS] \otimes \pi^* \xi 
\to \imath_* \mO _S(\bigoplus_{i=0}^{d-1} L^{-i} \otimes \xi) \to 0.
\end{split} 
\end{align}
Let $\lambda _d( \pi^* \xi)$, $\lambda _V([-kS] \otimes \pi^* \xi)$ 
$(k\geq 1)$
be the complexe lines
\begin{align}\label{eq:2.37} \begin{split}
&\lambda _d( \pi^* \xi) 
=\lambda(\oplus_{i=0}^{d-1} [-iS] \otimes \pi^* \xi)
\otimes \lambda ^{-1}(\oplus_{i=1}^{d} [-iS] \otimes \pi^* \xi),\\
&\lambda _V([-kS] \otimes \pi^* \xi)= \lambda([-(k-1)S] \otimes \pi^* \xi)
\otimes \lambda ^{-1}([-kS] \otimes \pi^* \xi).
\end{split}\end{align}
By  \cite{KnudsenMumford}, \eqref{eq:2.35}, \eqref{eq:2.36}, 
we have the canonical isomorphisms:
\begin{align} \label{eq:2.38}\begin{split}
 \lambda(\oplus_{i=0}^{d-1} L^{-i} \otimes \xi) 
& \simeq \lambda_d (\pi^* \xi),
\quad \lambda (L^{-k+1} \otimes \xi) 
\simeq \lambda _V([-kS] \otimes \pi^* \xi),\\
\lambda_d ( \pi^* \xi) &  = \lambda'_d ( \pi^* \xi).
\end{split}\end{align}
Let $\varphi_k, \rho_d$ be the canonical sections of 
$$\lambda^{-1} (L^{-k+1} \otimes \xi) 
\otimes  \lambda _V([-kS] \otimes \pi^* \xi), \quad 
 \lambda^{-1}(\oplus_{i=0}^{d-1} L^{-i} \otimes \xi) 
\otimes \lambda_d(\pi^* \xi).$$
Then
\begin{align}\label{eq:2.39}
 \rho_d = \bigotimes_{i=1}^{d} \varphi_i
\end{align}

\begin{prop} \label{p2.6} Under the identification \eqref{eq:2.12}, we have
\begin{align}\label{eq:2.40}
\tau_d =   \rho_d. 
\end{align}
\end{prop}

\begin{proof}
For $k\geq 1$, consider  the complex of $\mO _V$-sheaves on $V$
\begin{align}\label{eq:2.41}  \begin{array}{ll}
\hspace*{16mm}0 \hspace*{20mm}0 \hspace*{26mm}0 & 
\hspace*{4mm}0\\
0\to  \stackrel{\uparrow}{\imath_* \mO _S (L^{-k})} 
\stackrel{\mbox{\scriptsize Id}}{\to}
\imath_* \stackrel{\uparrow}{\mO _S} (\oplus_{i=0}^{k} L^{-i})
\to \stackrel{\uparrow}{\imath_*} \mO _S (\oplus_{i=0}^{k-1} L^{-i}) \to&
\hspace*{4mm}\stackrel{\uparrow}{ 0} \to 0\\
0\to \hspace*{8mm}\stackrel{\uparrow}{ 0} \to
\hspace*{13mm} \stackrel{\uparrow}{\mO _V} 
\stackrel{\mbox{\scriptsize Id}}{\to }
\hspace*{20mm}\stackrel{\uparrow}{\mO _V} \to & 
\hspace*{4mm}\stackrel{\uparrow}{0} \to 0\\
0\to \hspace*{8mm} \stackrel{\uparrow}{0} \to 
\hspace*{6mm}\stackrel{\tau_{[S]}^{k+1}\uparrow}{\mO _V[-(}k+1)S])
\stackrel{\tau_{[S]}}{ \to}
 \stackrel{\tau_{[S]}^k\uparrow}{\mO _V }([-k S]) \to 
& \stackrel{\uparrow}{\imath_* \mO _S} (L^{-k}) \to 0\\
\hspace*{15mm}\stackrel{\uparrow}{0}\hspace*{20mm} 
\stackrel{\uparrow}{0}
\hspace*{27mm} \stackrel{\uparrow}{0} &
\hspace*{4mm}\stackrel{\uparrow}{0}
\end{array}\end{align}
In \eqref{eq:2.41}, the rows  are exact sequences of sheaves. 
The second and third columns correspond to \eqref{eq:2.25}.

By \cite{KnudsenMumford}, \eqref{eq:2.29}, 
\eqref{eq:2.38} and \eqref{eq:2.41}, we have 
\begin{align}\label{eq:2.42} 
\tau^{-1}_{k} \otimes \tau_{k+1} = \varphi_{k+1}.
\end{align}
By \eqref{eq:2.25}, \eqref{eq:2.35}, we have also
\begin{align}\label{eq:2.43} 
\varphi_1= \tau_1.
\end{align}
By \eqref{eq:2.39}, \eqref{eq:2.42}, \eqref{eq:2.43}, 
we have \eqref{eq:2.40}.
\end{proof}

\section{Comparison formula for the Quillen metrics}\label{s3}

\begin{defn}\label{d3.1} Let $P^V$ be the vector space 
of smooth forms on a complex manifold $V$, which are sums of forms
of type $(p,p)$.
 Let $P^{V,0}$ be the vector space of the forms $\alpha \in P^V$ 
such that there exist smooth
 forms $\beta, \gamma$ on $V$ for which $\alpha= \partial \beta + 
\overline{\partial} \gamma$. 
\end{defn}

If $A$ is (q,q) matrix, set
\begin{align}\label{eq:3.1}
\td (A) =  \det\Big(\frac{A}{1- e^{-A}}\Big),\qquad
 \ch (A) =\tr[\exp(A)], \quad c_{1}(A) = \tr[A].
\end{align}
The genera associated to Td and ch  are called the Todd genus and the 
Chern character.

Let $P$ be an ad-invariant power series on square matrices. If $(F, h^F)$
is a holomorphic Hermitian vector bundle on $V$, let $ \nabla^F$ be 
the corresponding holomorphic Hermitian connection, 
and let $R^F$ be its curvature. Set 
\begin{align}\label{eq:3.1a}
P(F, h^F) = P\Big(\frac{-R^F}{2i\pi}\Big).
\end{align}
By the Chern-Weil theory, then $P(F,h^F)$ is a closed form which lies in $P^V$, and its cohomology
 class $P(F)$ does not depend  on $h^F$.

From now on, we use the assumption and notation 
of Section \ref{s1} and $S$ is a compact K\"ahler manifold.
Then $V$ is K\"ahler. 
Recall that we identify $S$ with $\{(x, (0,1))\in V: x\in S\}\subset V$.

Let $N_{S/V}$, $N_{W/V}$ be the normal bundles to $S,W$ in $V$.

Let $h^{TV}$ be a K\"ahler metric on $TV$. 
Let $h^{TW}, h^{TS}, h^{TY}$ be 
 the metrics on $TW, TS, TY$ induced by $h^{TV}$.
Let $h^{N_{S/V}}, h^{N_{W/V}}$ be the metrics on 
$N_{S/V}, N_{W/V}$,
as the orthogonal complements of $TS, TW$, induced by $h^{TV}$. 

By \eqref{eq:1.6}, \eqref{eq:2.4}, the maps
\begin{align}\label{eq:3.2a}\begin{split}
	N_{S/V}\to  [S]|_{S}, &\qquad N_{W/V}\to  [W]|_{W},\\
	y\to \partial_{y} \tau_{[S]},   & \qquad 
	\qquad y\to \partial_{y} \tau_{[W]}
\end{split}
\end{align}
define the canonical isomorphisms of 
	$N_{S/V}\simeq  [S]|_{S}, N_{W/V}\simeq [W]|_{W}$.
Let $h^{[S]}$ (resp. $h^{[W]}$) be a Hermitian metric on $[S]$ 
(resp. $[W]$) on $V$ such that the isomorphisms
(\ref{eq:3.2a}) are isometries.

 Let $h^{[-iS]}$ be the metrics on $[-iS]$ induced
 by $h^{[S]}$ and let $h^L$ be the metric on $L$ induced by 
 $h^{[S]}$ via (\ref{eq:1.7a}).
 Let $h^{[-W]}$ be the dual metric on $[-W]$ induced by $h^{[W]}$.

\comment{
Let $h^{[S]}$ be a Hermitian metric on $[S]$ on $V$. Let $h^{[-iS]}$, 
$h^{[-W]} = h^{[-d S]}$ be the metrics on $[-iS], [-W]= [-dS]$ induced
 by $h^{[S]}$. Let $h^L$ be the metric on $L$ induced by $h^{[S]}_{|S}$.

By \cite[p146]{GriffithsHarris}, we have also
\begin{align}\label{eq:3.2}
\begin{split}
	N_{S/V}= [S]|_{S},\quad N_{W/V}= [W]|_{W}.
	\begin{split}
\end{align}
Let $h^{N_{S/V}}, h^{N_{W/V}}$ be the metrics on $N_{S/V}, N_{W/V}$
induced by $h^{[S]}, h^{[W]}$. 
}

Let $h^{\xi}$ be a metric on $\xi$. Let $h^{\xi'}$ be the metric 
on $\xi'$ induced  by $h^{\xi}$. Let $h^{R\pi_{W*} \xi'}$ be the metric on 
$R^0\pi_{W*} \xi'$ induced by $h^L, h^\xi$ 
under identification \eqref{eq:2.12}.

Let $||\quad ||_{\lambda(\xi')}, ||\quad ||_{\lambda (R\pi_{W*}\xi')}$
be the Quillen metric \cite{Q85a}, \cite{BGS88c} on $\lambda(\xi')$, 
$\lambda (R\pi_{W*}\xi')$. Under identification \eqref{eq:2.12}, 
all the complex 
lines considered in Section \ref{s2} provide with the Quillen metrics.

Let $\zeta (s)$ be the Riemann zeta function. Let $R(x)$ be the 
Gillet-Soul\'e power series \cite{GS91},
\begin{align}\label{eq:3.3}
R(x)= \sum_{\stackrel{n\geq 1}{n\, \mbox{\scriptsize odd}}} 
\Big (\frac{2\zeta'(-n) }{\zeta(-n)} + \sum_{j=1}^n \frac{1}{ j}\Big )
 \zeta(-n) 
\frac{x^n}{ n!}.
\end{align}
We identify $R$ to the corresponding additive genus.

Let $P^V_W$ be the set of currents on $V$ which are sums of
 currents of type (p,p), whose wave front set is included in 
$N^*_{W/V,   \R }$. Let $P^{V,0}_W$ be the set of current 
$\alpha \in P^V_W$ such that there exist currents 
$\beta,\gamma$ on $V$, 
whose wave front set is included in $N^*_{W/V,   \R }$, such that 
$\alpha =  \partial \beta + \overline{\partial} \gamma$.

Let $(\xi_1, v), (\xi_2, v)$ be the complexes on $V$
\begin{align} \label{eq:3.4}\begin{split}
(\xi_1, v):&\quad 0 \to [-W]\otimes \pi^* \xi  
\stackrel{\tau_{[W]}}{\to} \pi^* \xi \to 0,\\
(\xi_2, v) :& \quad 0 \to \bigoplus_{i=0}^{d-1} [-(i+1)S] \otimes
\pi^* \xi \stackrel{\tau_{[S]}}{\to}
 \bigoplus_{i=0}^{d-1} [-iS] \otimes \pi^* \xi \to 0.
\end{split}\end{align}
Let $h^{\xi_1}$ (resp. $h^{\xi_2}$) be the metrics on $\xi_1$ (resp. 
on $\xi_2$) induced by  $h^{[W]}$ (resp. $h^{[S]}$)
and $h^{\xi}$.


Let $T(\xi_1, h^{\xi_1}) \in P^V_W,  T(\xi_2, h^{\xi_2})\in P^V_S$ 
be the Bott-Chern currents constructed in \cite[Theorem 2.5]{BGS90b}.  
The forms 
$T(\xi_i, h^{\xi_i})$  verify the following equations
\begin{align}\label{eq:3.5}\begin{split}
\frac{ \overline{\partial} \partial}{2i\pi}T(\xi_1, h^{\xi_1})
&= \td ^{-1}(N_{W/V}, h^{N_{W/V}}) \ch (\xi',h^{\xi'}) \delta_{\{W\}} 
-  \ch (\xi_1,h^{\xi_1})\\
&= \td ^{-1}([W], h^{[W]}) \pi^{*} \ch (\xi,h^{\xi}) 
\Big(\delta_{\{W\}} - c_{1}([W], h^{[W]})\Big) ,\\
\frac{ \overline{\partial} \partial}{2i\pi}T(\xi_2, h^{\xi_2})
&= \td ^{-1}(N_{S/V}, h^{N_{S/V}}) 
\sum_{j=0}^{d-1}\ch([-jS], h^{[-jS]})\ch (\xi,h^{\xi}) \delta_{\{S\}} 
\\
&\qquad-  \ch (\xi_2,h^{\xi_2})\\
&=\Big(\frac{1-e^{-dx}}{x}\Big) ([S], h^{[S]}) 
\pi^{*} \ch (\xi,h^{\xi}) 
\Big(\delta_{\{S\}} - c_{1}([S], h^{[S]})\Big) .
\end{split}\end{align}

Over $W$, we have the exact sequence of holomorphic Hermitian vector 
bundles
\begin{align}\label{eq:3.6}
0\to TW \to TV \to N_{W/V} \to 0.
\end{align}
Let $\widetilde{\td }(TW, TV|_{W}, h^{TV}) \in P^W/P^{W,0}$ 
be the Bott-Chern class constructed in  \cite[Theorem 1.29]{BGS88a}, 
such that 
\begin{align}\label{eq:3.7}
\frac{ \overline{\partial} \partial}{2i\pi}
\widetilde{\td }(TW, TV|_{W}, h^{TV})
= \td (TV, h^{TV}) 
- \td (TW, h^{TW}) \td (N_{W/V}, h^{N_{W/V}}).
\end{align}

Over $S$, we have the exact sequence of holomorphic Hermitian 
vector bundles
\begin{align}\label{eq:3.8}
0\to TS \to TV \to N_{S/V} \to 0.
\end{align}
Let $\widetilde{\td }(TS, TV|_{S}, h^{TV}) \in P^S/P^{S,0}$
be the corresponding  Bott-Chern class of \cite{BGS88a}. 
It verifies the following equation 
\begin{align}\label{eq:3.9}
\frac{ \overline{\partial} \partial}{2i\pi}
\widetilde{\td }(TS, TV|_{S}, h^{TV}) 
=  \td (TV, h^{TV}) 
- \td (TS, h^{TS}) \td (N_{S/V}, h^{N_{S/V}}).
\end{align}

\comment{
On $V$, we have also the exact sequence of holomorphic 
Hermitian vector bundles
\begin{align}\label{eq:3.10}
0\to TY \to TV \to \pi^* TS \to 0.
\end{align}
Let $\widetilde{\td }(TV, TS, h^{TV}, h^{TS})\in P^V/P^{V,0}$ 
be the corresponding  Bott-Chern class such that
\begin{align}\label{eq:3.11}
\frac{ \overline{\partial} \partial}{2i\pi}
\widetilde{\td }(TV, TS, h^{TV}, h^{TS})
=  \td (TV, h^{TV}) 
- \pi^* \Big (\td (TS, h^{TS}) \Big ) \td (TY, h^{TY}).
\end{align}

On $S$, the map  $TY \to N_{S/V}$ induced by $TV\to N_{S/V}$
is an isomorphism.
Let $\widetilde{\td }(TY|_{S}, h^{TY}, h^{N_{S/V}})$ $\in P^S/P^{S,0}$
be the Bott-Chern class such that
\begin{align}\label{eq:3.12}
	\frac{ \overline{\partial} \partial}{2i\pi}
\widetilde{\td }(TY|_{S}, h^{TY}, h^{N_{S/V}})
= \td (TY|_{S}, h^{TY}) - \td (N_{S/V},  h^{N_{S/V}}).
\end{align}
}

	We establish first the following result about Bott-Chern classes.
\begin{lemma} \label{d3.3} For  a holomorphic line  bundle $F$ on a 
compact complex manifold $Z$, and $h^{F}, h^{F}_{1}$ two metrics on 
$F$ with dual metrics $h^{F^{-1}}, h^{F^{-1}}_{1}$ on $F^{-1}$,
we have in $P^{Z}/P^{Z,0}$,
\begin{align}\label{eq:3.4a}
- \wi{\ch} (F^{-1},  h^{F^{-1}}, h^{F^{-1}}_{1}) =
\wi{\td^{-1}} (F, h^{F}, h^{F}_{1}) c_{1}(F, h^{F})
+ \td^{-1} (F, h^{F}_{1}) \wi{c_{1}} (F, h^{F}, h^{F}_{1}),
\end{align}
with $\wi{\ch}, \wi{\td^{-1}}, \wi{c_{1}}$ the Bott-Chern classes 
such that
\begin{align}\label{eq:3.4b}\begin{split}
&\frac{ \overline{\partial} \partial}{2i\pi}
\wi{\ch} (F^{-1},  h^{F^{-1}}, h^{F^{-1}}_{1})
= \ch (F^{-1}, h^{F^{-1}}_{1}) - \ch (F^{-1}, h^{F^{-1}}),\\
&\frac{ \overline{\partial} \partial}{2i\pi}
\wi{\td^{-1}} (F, h^{F}, h^{F}_{1})
= \td^{-1}(F, h^{F}_{1}) - \td^{-1}(F, h^{F}),\\
&\frac{ \overline{\partial} \partial}{2i\pi}
\wi{c_{1}} (F, h^{F}, h^{F}_{1})
= c_{1} (F,  h^{F}_{1})- c_{1} (F,  h^{F}) .
\end{split}\end{align}

	\end{lemma} 
\begin{proof} Consider two holomorphic line bundles $F, \eta$ 
with metrics $h^{F}, h^{\eta}$ and the characteristic form 
	\begin{align}\label{eq:3.5a}
\phi(h^{F}, h^{\eta}) = \td^{-1}(F, h^{F})\,  c_{1}(\eta, h^{\eta}).
\end{align}
Then by \cite[Remark 1.28]{BGS88a}, given another pairs of metrics
$h^{F}_{1}, h^{\eta}_{1}$ on $F, \eta$, the associated Bott-Chern
form $\wi{\phi}$ is given by : for any smooth path $h^{F}_{t},
h^{\eta}_{t}$ for $t\in [0,1]$, from $h^{F}, h^{\eta}$ to 
$h^{F}_{1}, h^{\eta}_{1}$ respectively, the form
\begin{multline}\label{eq:3.6a}
\wi{\phi} = \int_{0}^{1}\Big\{\Big(\frac{1-e^{-x}}{x}\Big)' (F,  h^{F}_{t}) 
 (h^{F}_{t})^{-1} \frac{\partial h^{F}_{t}}{\partial t}
 c_{1} (\eta, h^{\eta}_{t})\\
 + \Big(\frac{1-e^{-x}}{x}\Big) (F, 
 h^{F}_{t}) (h^{\eta}_{t})^{-1} \frac{\partial h^{\eta}_{t}}{\partial t}
\Big\} dt \in P^{Z}/P^{Z,0}
\end{multline}
does not depend on the choice of the path 
$h^{F}_{t}, h^{\eta}_{t}$ and 
\begin{align}\label{eq:3.7a}
\frac{ \overline{\partial} \partial}{2i\pi}\wi{\phi} 
=  \phi(h^{F}_{1}, h^{\eta}_{1}) - \phi(h^{F}, h^{\eta}).
\end{align}

In particular, if we choose a path such that $h^{\eta}_{t}=h^{\eta}$
for $t\in [0, \frac{1}{2}]$ and $h^{F}_{t}=h^{F}_{1}$
for $t\in [ \frac{1}{2}, 1]$, then we get
\begin{align}\label{eq:3.8a}\begin{split}
& \wi{\phi} =\wi{\td^{-1}} (F, h^{F}, h^{F}_{1}) c_{1}(\eta, h^{\eta})
+ \td^{-1} (F, h^{F}_{1}) \wi{c_{1}} (\eta, h^{\eta}, h^{\eta}_{1}),\\
&\wi{c_{1}} (\eta, h^{\eta}, h^{\eta}_{1})
= \log \frac{h^{\eta}_{1}}{h^{\eta}}.
\end{split}\end{align}
If we take $\eta=F$ and $h^{F}_{t}= h^{\eta}_{t}$, the we get
\begin{align}\label{eq:3.9a}
\phi(h^{F}_{t}, h^{F}_{t}) = 1- \ch(F^{-1}, h^{F^{-1}}_{t}),
\end{align}
thus in this case
\begin{align}\label{eq:3.10a}
 \wi{\phi} = - \wi{\ch} (F^{-1},  h^{F^{-1}}, h^{F^{-1}}_{1}) .
\end{align}
From  \eqref{eq:3.8a} and \eqref{eq:3.10a}, 
we get \eqref{eq:3.5a}.
\end{proof}

We define 
\begin{multline}\label{eq:3.12a}
 \mT ( h^{[S]}, h^{[W]}) =  \td^{-1}([W], h^{[W]}) \log 
 \|\tau_{[W]}\|^{2}_{h^{[W]}}\\
 - \td^{-1}([dS], h^{[dS]}) \log 
 \|\tau_{[S]}^{d}\|^{2}_{h^{[dS]}}
 - \wi{\ch} ([-dS], h^{[-dS]}, h^{[-W]}).
\end{multline}
\begin{lemma} \label{d3.4} In $P^{V}_{W\cup S}/P^{V,0}_{W\cup S}$,
	$ \mT ( h^{[S]}, h^{[W]})$ does not depend on the choice of 
	$h^{[S]}, h^{[W]}$, thus we denote it as $\mT_{S,W}$,
	and we have 
	\begin{align}\label{eq:3.13a}
\frac{ \overline{\partial} \partial}{2i\pi}\mT_{S,W}
= \td^{-1}(N_{W/V}, h^{N_{W/V}}) \delta_{\{W\}}
- \Big (\frac{1-e^{-dx}}{x}\Big) (N_{S/V}, h^{N_{S/V}}) \delta_{\{S\}}.
\end{align}
\end{lemma} 

\begin{proof} By Poincar\'e-Lelong formula and (\ref{eq:3.2a}), we 
get first  (\ref{eq:3.13a}).

Let $h^{[W]}_{1}$ be another metric on $[W]$ such that  
(\ref{eq:3.2a}) is an isometry.  Then by Lemma \ref{d3.3}, we have in 
$P^{V}_{W\cup S}/P^{V,0}_{W\cup S}$,
\begin{multline}\label{eq:3.14a}
 \mT ( h^{[S]}, h^{[W]}) -  \mT ( h^{[S]}, h^{[W]}_{1}) \\
 = \Big(\td^{-1}([W], h^{[W]}) - \td^{-1}([W], h^{[W]}_{1})\Big)
 \log  \|\tau_{[W]}\|^{2}_{h^{[W]}} \\
 - \td^{-1}([W], h^{[W]}_{1}) \log \frac{h^{[W]}_{1}}{h^{[W]}}
 + \wi{\ch} ([-dS], h^{[-W]}, h^{[-W]}_{1})\\
 = \wi{\td^{-1}} ([W], h^{[W]}_{1}, h^{[W]})\delta_{\{W\}}=0,
\end{multline}
as $h^{[W]}= h^{[W]}_{1}= h^{N_{W/V}}$ on $W$.

By the same argument, we know also
$\mT ( h^{[S]}, h^{[W]})$ does not depend on $h^{[S]}$.
\end{proof}
	
\begin{thm} \label{t3.2} The following identity holds
\begin{multline}\label{eq:3.13}
\log (||\sigma|| ^2_{\lambda(\xi') \otimes 
\lambda^{-1}(R\pi_{W*}\xi')})
=\int_{V} \td (TV, h^{TV}) \mT_{S,W}\\
\qquad - \int_W \td ^{-1} (N_{W/V}, h^{N_{W/V}})
 \widetilde {\td }(TW, TV|_{W}, h^{TV}) 
 \ch  (\xi', h^{\xi'})\\
 +\int_S \Big(\frac{1-e^{-dx}}{ x}\Big) (L, h^{[S]}) 
 \widetilde{\td }(TS, TV|_{S}, h^{TV}) \ch (\xi,h^{\xi}) \\
 +\int_S \td (TS)R(TS) \ch (R^{\bullet} \pi_{W*} \mO _W)
\ch (\xi)
-\int_{W}\td (TW) R(TW) \ch (\xi').
\end{multline}
\end{thm}
\begin{proof}
	Let $\|\quad \|_{\lambda'_d(\pi^* \xi)}^{2}$ be the Quillen 
	metric on $\lambda'_d(\pi^* \xi)$ (\ref{eq:2.28}) induced by 
	$h^{[W]}$, $h^{\xi}$ and $h^{TV}$. 
	Let $\|\quad \|_{\lambda_d(\pi^* \xi)}^{2}$ be the Quillen 
	metric on $\lambda_d(\pi^* \xi)\simeq\lambda'_d(\pi^* \xi)$
(\ref{eq:2.37}) induced by $h^{[S]}$, $h^{\xi}$ and $h^{TV}$. 
	By the anomaly formula \cite[Theorem 1.23]{BGS88c}, we have 
	\begin{align}\label{eq:3.20a}
\log \frac{\|\quad \|_{\lambda'_d(\pi^* \xi)}^{2}}
{\|\quad \|_{\lambda_d(\pi^* \xi)}^{2}}
= - \int_{V} \td (TV, h^{TV}) 
\wi{\ch}( [-dS], h^{[-dS]}, h^{[-W]}).
\end{align}

By using 
\cite[Theorem 6.1]{BL91},
 \eqref{eq:2.13}, \eqref{eq:2.36} and \eqref{eq:3.4}, we have 
\begin{multline}\label{eq:3.14}
\log (||\sigma_1||^2_{\lambda'_d(\pi^* \xi) \otimes \lambda^{-1}(\xi')})=
- \int _V \td (TV, h^{TV}) T(\xi_1, h^{\xi_1})\\
+ \int_W \td ^{-1}(N_{W/V}, h^{N_{W/V}}) 
\ch (\xi',h^{\xi'})
\widetilde{\td }(TW, TV|_{W}, h^{TV})\\
 -\int_{V}\td (TV) R(TV) \ch (\xi) (1-\ch ([-W]))
 +\int_{W}\td (TW) R(TW) \ch (\xi'),
\end{multline}
\begin{multline*}
\log (||\rho_d||^2_{\lambda_d(\pi^* \xi)  \otimes 
\lambda^{-1}(R^{\bullet} \pi_{W*}\xi')})=
- \int _V \td (TV, h^{TV}) T(\xi_2, h^{\xi_2})\\
+\int_S \td ^{-1}(N_{S/V}, h^{N_{S/V}})
\widetilde{\td }(TS, TV|_{S}, h^{TV})
 \ch (R^{\bullet} \pi_{W*}\xi', 
\oplus_i h^{L^{-i}}\otimes h^{\xi})\\
-\int_{V}\td (TV) R(TV) \ch (\xi) (1-\ch ([-dS]))
 +\int_S \td (TS)R(TS) 
\ch (R^{\bullet} \pi_{W*}\xi').
\end{multline*}

By \cite[Remark 3.5 and Theorem 3.17]{BGS90a}, 
\begin{align}\label{eq:3.16a}
\begin{split}
T(\xi_1, h^{\xi_1})&= \pi^*  ( \ch (\xi,h^{\xi}))
\td ^{-1}([W], h^{[W]}) \log ||\tau_{[W]}||^2_{h^{[W]}}
 \quad   \text{in} \quad P^V_W/P^{V,0}_W ,\\
T(\xi_2, h^{\xi_2})&= \pi^*  ( \ch (\xi,h^{\xi}))
\ch (\oplus_{i=0}^{d-1} [-iS], \oplus h^{[-iS]})
\td ^{-1}([S], h^{[S]}) \log ||\tau_{[S]}||^2_{h^{[S]}}\\
&= \pi^*  ( \ch (\xi,h^{\xi}))\td ^{-1}([dS], h^{[dS]}) 
\log ||\tau_{[S]}^{d}||^2_{h^{[dS]}} 
\quad  \text{in} \quad P^V_S/P^{V,0}_S.
\end{split}
\end{align}


By \eqref{eq:1.7a}, \eqref{eq:2.12} and \eqref{eq:3.2a}, we have
\begin{align}\label{eq:3.17a}
\td ^{-1}(N_{S/V}, h^{N_{S/V}})
 \ch (R^{\bullet} \pi_{W*}\xi', \oplus_i h^{L^{-i}}\otimes h^{\xi})
= \Big(\frac{1-e^{-dx}}{ x}\Big) (L, h^{[S]})  \ch (\xi,h^{\xi}).
\end{align}

By Propositions \ref{p2.5}, \ref{p2.6}, 
and our identification of $\lambda_d(\pi^* \xi)$
to $\lambda'_d(\pi^* \xi)$ by \eqref{eq:2.32}, we  have 
\begin{multline}\label{eq:3.17}
||\sigma||^2_{\lambda(\xi') \otimes \lambda^{-1}(R\pi_{W*}\xi')}
=
(||\sigma_1||^2_{\lambda'_d(\pi^* \xi) \otimes 
\lambda^{-1}(\xi')})^{-1}\\
\cdot ||\rho_d||^2_{\lambda_d(\pi^* \xi) \otimes 
\lambda^{-1}(R^{\bullet} \pi_{W*}\xi')}
\frac{\|\quad \|_{\lambda'_d(\pi^* \xi)}^{2}}
{\|\quad \|_{\lambda_d(\pi^* \xi)}^{2}}.
\end{multline}

From Lemma \ref{d3.4},  \eqref{eq:3.20a}-\eqref{eq:3.17}, 
we deduce \eqref{eq:3.13}.  
\end{proof}

\begin{rem} \label{t3.3a}  From $V= \PP(L\oplus 1)$, as holomorphic 
	vector bundles on $S$, we have
\begin{align}\label{eq:3.18a}
TV|_{S}= TS \oplus L, \quad \text{and }  TY|_{S}= L \simeq N_{S/V}.
\end{align}
Starting from a metric on $L$, by using the first Chern form of 
$\mO_{V}(1)$ and a K\"ahler metric on $S$, we can construct a 
K\"ahler metric on $V$ such that  (\ref{eq:3.18a}) is an isometry 
with induced metrics on $TS$, $TY$.
Under this assumption, 
 (\ref{eq:3.8}) splits with metrics as in
 (\ref{eq:3.18a}), thus
 \begin{align}\label{eq:3.19a}
\widetilde{\td }(TS, TV|_{S}, h^{TV}) =0.
\end{align}
	\end{rem} 
\section{Comparison formula for equivariant Quillen metrics}\label{s4}

$\quad$ In the sequel, we suppose that for $1\leq i \leq d-1$, 
$\alpha_i =0$ in \eqref{eq:1.1}. So
\begin{align}\label{eq:4.1}
W= \Big \{ (x,t)\in L: t^d + \alpha_d (x) =0 \Big \}.
\end{align}

Let $G= { \Z /d \Z } 
= \{0, \overline 1, \cdots, \overline{d-1} \}$. In this case, the group $G$
acts naturally on $V$. The action of $G$ is defined by : 
for $g=\overline 1$,
$(t,u) \in L\oplus \C$, the homogeneous coordinate of $V$.
\begin{align}\label{eq:4.2}
g \cdot (t,u) = (e^{i {2\pi/d}} t,u).
\end{align}
Then $G$ preserves $W$, and $S= W/G$.
Let $G$ act on $\mO _V$ by
\begin{align}
g \cdot f(\cdot)= f(g^{-1}\cdot),\quad   \text{for}\quad  g\in G,\, 
f\in \mO _V. \nonumber
\end{align}

Let $G$ act trivially on $\xi$. Then $G$ acts also on $\xi'$.
Let  $G$ act on $L$ by  following : for $g= \overline 1$, $t\in L$,
\begin{align}\label{eq:4.3}
g \cdot t = e^{i {2\pi /d}} t.
\end{align}
Then it induces also an action on $L^{-i}, \pi^* L$.

If given $W\in \widehat{G}$,  $\lambda_W, \mu_W$ are complex lines, if 
$\lambda= \oplus_{W\in \widehat{G}} \lambda_W$, 
$\mu= \oplus_{W\in \widehat{G}} \mu_W$, set
\begin{align}\label{eq:4.4}
\lambda^{-1}= \bigoplus_{W\in \widehat{G}} \lambda^{-1}_W,\quad 
\lambda \otimes \mu 
= \bigoplus_{W\in \widehat{G}} \lambda_W \otimes \mu_W. 
\end{align}

Let $ \lambda_G(\xi')$, $\lambda_G(R^{\bullet}\pi_{W*} \xi')$ 
be the inverse 
 of the equivariant determinant of the cohomology of  
$\xi'$ and $R^\bullet  \pi _*\xi'$ on $W$, $S$ \cite[\S 2]{B95}. 
Then $ \lambda_G(\xi')$ (resp. $\lambda_G(R^{\bullet}\pi_{W*} \xi')$)  
is a direct sum of complex lines.  As in \cite{B95}, \cite{KnudsenMumford}, 
we  have a canonical isomorphism of direct sums of complex lines
\begin{align}\label{eq:4.5}
\lambda_G(\xi') \simeq \lambda_G(R^{\bullet}\pi_{W*} \xi').
\end{align}
Let $\sigma_G$ be the canonical nonzero section of 
$\lambda_G(\xi') \otimes  \lambda_G^{-1} (R^{\bullet}\pi_{W*} \xi')$.

Let $h^{TV}$ be a $G$-invariant K\"ahler metric on $V$
(cf. Remark \ref{t3.3a} for the existence).
  
We provide the $G$-invariant Hermitian metrics 
 $h^{[S]}$, $h^{[W]}$, $h^{\xi}$ 
on $ [S], [W],\xi$ such that (\ref{eq:3.2a}) are isometries. 
Then they determine the $G$-equivariant Quillen 
metrics $||\quad ||_{\lambda_G(\xi')}$, 
$||\quad ||_{ \lambda_G(R^{\bullet}\pi_{W*} \xi')}$ on the equivariant 
determinants $ \lambda_G(\xi')$, $\lambda_G(R^{\bullet}\pi_{W*} \xi')$
\cite[\S 2a)]{B95}.

By our constructions, \eqref{eq:2.13}, \eqref{eq:2.25}, \eqref{eq:2.35} 
are $G$-equivariant 
exact sequences of sheaves. And the splits of \eqref{eq:2.14}, 
\eqref{eq:2.26} are also $G$--equivariant. 
Set 
\begin{align}\label{eq:4.6}\begin{split}
\lambda'_{d,G}(\pi^* \xi)&= \lambda_G(\pi^* \xi) \otimes 
\lambda ^{-1}_G([-W]\otimes \pi^* \xi), \\
\lambda_{d,G}(\pi^* \xi) 
&= \lambda_G(\oplus_{i=1}^{d} [-iS] \otimes \pi^* \xi)
\otimes \lambda ^{-1}_G(\oplus_{i=0}^{d-1} [-iS] \otimes \pi^* \xi).
\end{split}\end{align}
 As in \cite{KnudsenMumford}, \cite[\S 3b)]{B95}, 
 by \eqref{eq:2.13}, \eqref{eq:2.25}, \eqref{eq:2.35}, we have 
 the canonical isomorphisms of direct sums of complex lines
\begin{align}\label{eq:4.7}
\lambda_G(\xi')\simeq  \lambda'_{d,G}(\pi^* \xi),
\quad  \lambda_G(R^{\bullet}\pi_{W*} \xi')
=  \lambda_G(\oplus_{i=0}^{d-1} L^{-i} \otimes \xi)
\simeq  \lambda_{d,G}(\pi^* \xi).
\end{align}
Let $\sigma_{1,G}, \rho_{d,G}$ be the canonical sections of 
$\lambda_G^{-1}(\xi') \otimes  \lambda'_{d,G}(\pi^* \xi)$, 
$\lambda_G ^{-1}(\oplus_{i=0}^{d-1} L^{-i} \otimes \xi) 
\otimes \lambda_{d,G}(\pi^* \xi)$.

We denote $\Sigma = W \cap S= \{x\in S :  \alpha_d(x) =0\}$.
As we suppose that $W$ is a manifold, we know that $\Sigma$
is a submanifold of $S$ and $\partial \alpha_{d}(x)\neq 0$
for any $x\in \Sigma$.
If $g\in G$, set 
\begin{align}\label{eq:4.8}
V^g=\{x\in V: gx=x\},\qquad
W^g = \{x\in W: g x =x\}.
\end{align}
If $g\neq \overline{0}$, then $V^g=S \cup \PP(L)$, $W^g= \Sigma$.

Let $\td _g (TV, g^{TV})$ be the Chern-Weil Todd form on $V^g$ 
associated to the holomorphic Hermitian connection on
 $(TV, h^{TV})$ \cite[\S 2a)]{B95}, which appears in the Lefschetz
 formulas of Atiyah-Bott \cite{ABo67}. Other Chern-Weil form
 will be denoted in a similar way. 
In particular, the forms $\ch _g(\xi_1, h^{\xi_1})$ on $V^g$ 
is the Chern-Weil representative of the $g$-Chern character form of 
$(\xi_1, h^{\xi_1})$. Also, we denote by 
$\td _g (TV), \ch _g(\xi_1) \cdots $
the cohomology classes of $\td _g (TV, g^{TV})$, 
$\ch _g(\xi_1, h^{\xi_1})$ $\cdots$ on $V^g$.

Let $R(\theta, x)$ be the power series in \cite[(7.39)]{B94},
 \cite[(7.43)]{B95}, which verifies $R(0, x)= R(x)$. 
Let $R_g(TV), \cdots$ be the corresponding 
additive genera  \cite[\S 7c)]{B94}, \cite[\S 7g)]{B95}.

Let $h^{T\Sigma}$
be the metric on $T\Sigma$ induced by $h^{TS}$.
Let $h^{N_{\Sigma/S}}$ 
be the metrics on $N_{\Sigma/S}$ induced by $h^{TV}$.
As smooth vector bundles on $\Sigma$, we have the $G$-equivariant 
orthogonal splitting 
\begin{align}\label{eq:4.8a}
TV|_{\Sigma} = T\Sigma\oplus N_{\Sigma/S}\oplus N_{S/V}
= T\Sigma \oplus N_{\Sigma/W}\oplus N_{W/V},
\end{align}
as $G$ acts trivially on $T\Sigma$, $N_{\Sigma/S}$, and nontrivially 
on $N_{S/V}, N_{\Sigma/W}$, we conclude that
\begin{align}\label{eq:4.9a}\begin{split}
N_{\Sigma/S}= N_{W/V}, \quad
h^{N_{\Sigma/S}}= h^{N_{W/V}}, \quad 
N_{\Sigma/W}= N_{S/V}, \quad
h^{N_{\Sigma/W}}= h^{N_{S/V}} \quad \text{on } \Sigma.
\end{split}\end{align}


\begin{thm}\label{t4.1} For 
 $ g = \overline{j}\,  (0< j \leq d-1)$, the following identity holds 
\begin{multline}\label{eq:4.11}
\log (||\sigma_G|| ^2_{\lambda_G(\xi')
 \otimes \lambda^{-1}_G(R\pi_{W*}\xi')})(g)\\
=\int_S \td (TS, h^{TS}) \td_{g} (N_{S/V}, h^{N_{S/V}})
\td ^{-1}([W], h^{[dS]})\ch (\xi,h^{\xi})
\log ||\alpha_d||^2_{h^{[dS]}}\\
-\int_{\Sigma} \td ^{-1}(N_{W/V}, h^{N_{W/V}}) 
\td _g(N_{S/V}, h^{N_{S/V}})
\widetilde{\td }(T\Sigma, TS|_{\Sigma}, h^{TS})\ch (\xi,h^{\xi})\\
+ \int_{\Sigma} \td (TS, h^{TS})\td_{g} (N_{S/V}, h^{N_{S/V}})
\widetilde{\td^{-1} }(N_{W/V}, h^{[dS]}, h^{N_{W/V}})\ch 
(\xi,h^{\xi})\\
+ \int_S \td (TS)R(TS)\ch (\xi) 
\ch _g(R^{\bullet} \pi_{W*} \mO _W)
-\int_{\Sigma}\td _g(TW) R_g(TW) \ch (\xi).
\end{multline}
\end{thm}

\begin{proof}
		By the anomaly formula \cite[Theorem 2.5]{B95}, we have 
	\begin{align}\label{eq:4.11a}
\log \Big(\frac{\|\quad \|_{\lambda'_d(\pi^* \xi)}^{2}}
{\|\quad \|_{\lambda_d(\pi^* \xi)}^{2}}\Big) (g)
= - \int_{S\cup \PP(L)} \td_{g} (TV, h^{TV}) 
\wi{\ch}_{g}( [-dS], h^{[-dS]}, h^{[-W]}).
\end{align}

 By applying  \cite[Theorem 0.1]{B95} to \eqref{eq:2.13}, \eqref{eq:2.36}, 
 we have
\begin{multline}\label{eq:4.12}
\log (||\sigma_{1,G}||^2
_{\lambda'_{d,G}(\pi^* \xi) \otimes \lambda^{-1}_G(\xi')})(g)=
- \int _{S\cup \PP(L)} \td _g(TV, h^{TV}) T_g(\xi_1, h^{\xi_1})\\
+ \int_{\Sigma} \td ^{-1}_g(N_{W/V}, h^{N_{W/V}}) 
\ch _g(\xi,h^{\xi})
\widetilde{\td }_g(TW|_{\Sigma}, TV|_{\Sigma}, h^{TV})\\
 -\int_{S\cup \PP(L)}  \td _g(TV) R_g(TV) \pi^{*}\ch (\xi) 
 (1-\ch_{g}([-W]))
+\int_{\Sigma}\td _g(TW) R_g (TW) \ch _g(\xi).
\end{multline}
\begin{multline*}
\log (||\rho_{d,G}||^2_{\lambda_{d,G}(\pi^* \xi) \otimes 
\lambda^{-1}_G(R^{\bullet} \pi_{W*}\xi)})(g)= 
- \int _{S\cup \PP(L)} \td _g(TV, h^{TV}) T_g(\xi_2, h^{\xi_2})\\
+\int_S \td ^{-1}_g(N_{S/V}, h^{N_{S/V}})
\widetilde{\td }_g(TS, TV|_{S}, h^{TV})
 \ch (\xi, h^{\xi})
 \ch _g(R^{\bullet} \pi_{W*} \mO _W, \oplus h^{[-iS]})\\
 -\int_{S\cup \PP(L)}  \td _g(TV) R_g(TV) \pi^{*}\ch (\xi) 
 (1-\ch_{g}([-dS]))\\
+\int_{S}  \td _g(TS)R_g(TS) \pi^{*}\ch (\xi) 
\ch _g(R^{\bullet} \pi_{W*} \mO _W).
\end{multline*}

In this case, since the identifications in Section \ref{s2}  is $G$-equivariant,
 as (\ref{eq:3.17}), we have
\begin{multline}\label{eq:4.13}
 \|\sigma_{G}\|^2_{\lambda_G(\xi')
 \otimes \lambda^{-1}_G(R\pi_{W*}\xi')}(g) 
 = \Big\{(\|\sigma_{1,G}\|^2_{\lambda' _{d,G}(\pi^* \xi) \otimes 
\lambda^{-1}_G(\xi')})^{-1}\\
\cdot \|\rho_{d,G}\|^2_{\lambda_{d,G}(\pi^* \xi) \otimes 
\lambda^{-1}_G(R^{\bullet} \pi_{W*}\xi')}
\frac{\|\quad \|_{\lambda'_d(\pi^* \xi)}^{2}}
{\|\quad \|_{\lambda_d(\pi^* \xi)}^{2}} \Big\}(g). 
\end{multline}

Note that by (\ref{eq:1.6}), $g=\ov{j}$ acts on $[S]$ as multiplication by 
$e^{i2\pi j/d}$ on $S= V^{g}\cap S$, and $g$ acts on $[W]|_{S}$
as $g^{d}= \Id$.
By \cite[\S 6b)]{B95}, \cite[Theorem 3.17]{BGS90a}, 
\eqref{eq:3.4} and  \eqref{eq:4.8}, 
 on $S= V^g \cap S$, we calculate easily 
\begin{align}\label{eq:4.14}\begin{split}
T_g(\xi_1, h^{\xi_1})=&  \ch (\xi,h^{\xi})
\td ^{-1}([W], h^{[W]}) \log ||\alpha_d||^2_{h^{[W]}} 
 \qquad \text{in }\quad  P^S_{\Sigma}/P^{S,0}_{\Sigma}, \\
T_g(\xi_2, h^{\xi_2})=& 0  \qquad \text{in }\quad P^S/P^{S,0}.
\end{split}\end{align}
In the second equation of (\ref{eq:4.14}), we use $\tau_{[S]}=0$ on
$S^{g}= V^{g}\cap S$, thus the form 
$\tr_{s}[g N_{H} \exp(-C_{u}^{2})]$ in the definition 
of $T_g(\xi_2, h^{\xi_2})$ does not depend on $u$, and automatically 
$T_g(\xi_2, h^{\xi_2})$ vanishes.

As explain above,  on $S\cup \PP(L)$, $g$ acts as identity on $[dS]= [W]$,
and  by the argument in (\ref{eq:3.14a}), we know 
in $P^S_{\Sigma}/P^{S,0}_{\Sigma}$,
\begin{multline}\label{eq:4.14a}
\td ^{-1}([W], h^{[W]}) \log ||\alpha_d||^2_{h^{[W]}} 
-\wi{\ch}_{g}( [-dS], h^{[-dS]}, h^{[-W]})\\
- \td ^{-1}([W], h^{[dS]}) \log ||\alpha_d||^2_{h^{[dS]}} \\
= \wi{\td^{-1}} ( [W],    h^{[dS]}, h^{[W]} )\delta_{\{\Sigma\}}
= \wi{\td^{-1}} ( [W],   h^{[dS]}, h^{N_{W/V}} ) \delta_{\{\Sigma\}}.
\end{multline}

On $\PP(L)$, by \eqref{eq:1.6}, $g$ acts on $[S]$ as identity, and 
by \eqref{eq:1.5},  we have 
\begin{align}\label{eq:4.16}
[-W]= [-S]= \mO _{\PP(L)}  \quad \text{ on } \PP(L).
\end{align}
By using  \cite[\S 6b)]{B95},  we have also
\begin{align}  \label{eq:4.15} \begin{split}
T_g(\xi_1, h^{\xi_1})&=
 \widetilde{\ch }\Big([-W]|_{\PP(L)},h^{\mO _V}, h^{[-W]} \Big)
\ch (\xi,h^{\xi})  \quad \text{in } \,\, P^{\PP(L)}/P^{\PP(L),0},\\
T_g(\xi_2, h^{\xi_2})&=\sum_{i=1}^d 
\wi{\ch}\Big([-iS]|_{\PP(L)},  h^{[-(i-1)S]},h^{[-iS]}\Big)
\ch (\xi,h^{\xi})
\quad \text{in } \,\, P^{\PP(L)}/P^{\PP(L),0}.
\end{split}\end{align}
 From \eqref{eq:4.16}, we have 
\begin{multline}\label{eq:4.17}  
\sum_{i=1}^d \widetilde{\ch }\Big([-iS]|_{\PP(L)},
 h^{[-(i-1)S]},  h^{[-iS]}\Big)
+ \wi{\ch}_{g}( [-dS], h^{[-dS]}, h^{[-W]})\\
 =\widetilde{\ch }\Big([-W]|_{\PP(L)}, h^{\mO _V},  h^{[-W]}\Big)
  \quad \text{in } \,\, P^{\PP(L)}/P^{\PP(L),0}.
\end{multline}

Since $g$ acts on $N_{W/V}= [W]$ on $\Sigma$ as $\Id$, we have 
\begin{align}\label{eq:4.18}
 \td ^{-1}_g(N_{W/V}, h^{N_{W/V}})
 = \td ^{-1}(N_{W/V}, h^{N_{W/V}})
\quad\mbox{on} \quad \Sigma.
\end{align}

The restriction of the exact sequence \eqref{eq:3.6} on $\Sigma$
is split as in  \cite[(6.8)]{B95} to two  following exact sequences 
\begin{align}\label{eq:4.19}\begin{split}
0\to T\Sigma \to TS \to N_{\Sigma/S} \to 0,\quad
0\to N_{\Sigma/W} \to N_{S/V} \to 0 \to 0.
\end{split}\end{align}
By \eqref{eq:3.5}, \eqref{eq:4.9a} and \eqref{eq:4.19},
we have   
\begin{align}\label{eq:4.20}\begin{split}
\widetilde{\td }_g(TW|_{\Sigma}, TV|_{\Sigma}, h^{TV|_{\Sigma}}) 
&= \td _g(N_{S/V}, {h}^{N_{S/V}}) 
\widetilde{\td }(T\Sigma, TS|_{\Sigma}, h^{TS})
\quad \text{in} \,\, P^{\Sigma}/ P^{\Sigma,0},\\
 \td _g(TV, h^{TV}) &= \td _g(N_{S/V}, {h}^{N_{S/V}})
\td (TS, h^{TS}) \quad \text{on} \quad S.
\end{split}\end{align}

As \eqref{eq:3.8} splits $G$-equivariantly and isometrically, as 
in (\ref{eq:3.18a}), we get
\begin{align}\label{eq:4.22}
 \widetilde{\td }_g(TS, TV|_{S}, h^{TV}) = 0 
\quad \text{in} \quad P^{S}/P^{S,0}. 
\end{align}

By \eqref{eq:4.11a}-\eqref{eq:4.22}, we have \eqref{eq:4.11}.  
\end{proof}


\def\cprime{$'$} \def\cprime{$'$}
\providecommand{\bysame}{\leavevmode\hbox to3em{\hrulefill}\thinspace}
\providecommand{\MR}{\relax\ifhmode\unskip\space\fi MR }
\providecommand{\MRhref}[2]{%
  \href{http://www.ams.org/mathscinet-getitem?mr=#1}{#2}
}
\providecommand{\href}[2]{#2}

\end{document}